\documentclass[11pt]{article}
\usepackage{amsmath,amsthm,amsfonts,amssymb,mathrsfs,bm,wasysym}
\usepackage{subfigure}
\usepackage{graphicx}
\usepackage[usenames]{color}
\usepackage{verbatim}
\usepackage{epsfig}
\usepackage{authblk}

\parskip \medskipamount
\newtheorem{theorem}{Theorem}[section]
\newtheorem{definition}[theorem]{Definition}

\numberwithin{equation}{section}
\newtheorem{lemma}[theorem]{Lemma}
\newtheorem{proposition}[theorem]{Proposition}
\newtheorem{remark}[theorem]{Remark}

\newtheorem{claim}[theorem]{Claim}

\numberwithin{equation}{section}

\def\Z{\mathbb{Z}}

\def\R{\mathbb{R}}

\def\F{\mathcal{F}}

\def\B{\mathcal{B}}

\renewcommand{\phi}{\varphi}
\renewcommand{\epsilon}{\varepsilon}

\def\RR{\mathcal{R}}
\def\val{\textrm{Val}}

\def\G{\mathcal{G}}
\def\H{\mathcal{H}}

\allowdisplaybreaks

\newcommand{\1}{{\text{\Large $\mathfrak 1$}}}

\newcommand{\var}{\operatorname{var}}

\newcommand{\vol}{\mathrm{vol}}

\newcommand{\til}{\widetilde}

\newcommand{\pr}[1]{\mathbb{P}\!\left(#1\right)}
\newcommand{\E}[1]{\mathbb{E}\!\left[#1\right]}

\newcommand{\prstart}[2]{\mathbb{P}_{#2}\!\left(#1\right)}
\newcommand{\prcond}[3]{\mathbb{P}_{#3}\!\left(#1\;\middle\vert\;#2\right)}
\newcommand{\econd}[2]{\mathbb{E}\!\left[#1\;\middle\vert\;#2\right]}

\newcommand{\tn}{|\kern-.1em|\kern-0.1em|}

\newcommand\be{\begin{equation}}
\newcommand\ee{\end{equation}}

\begin{document}
\title{\bf   Hunter,   Cauchy Rabbit, and Optimal Kakeya Sets}

\author[1]{Yakov Babichenko}
\author[2]{Yuval Peres}
\author[3]{Ron Peretz}
\author[4]{Perla Sousi}
\author[5]{Peter Winkler}
\affil[1]{The Hebrew University of Jerusalem, Israel}
\affil[2]{Microsoft Research, Redmond, WA}
\affil[3]{Tel Aviv University, Israel}
\affil[4]{University of Cambridge, Cambridge, UK}
\affil[5]{Dartmouth College, Hanover, NH}
%
\date{}
\maketitle
\thispagestyle{empty}

\begin{abstract}
A planar set that contains a unit segment in every direction is called a Kakeya set. We relate these sets to a game of pursuit on a cycle $\Z_n$.
A hunter and a rabbit move on the nodes of $\Z_n$ without seeing each other.  At each step, the hunter moves to a neighbouring vertex or stays in place, while the rabbit is free to jump to any node. Adler et al (2003) provide strategies for hunter and rabbit that are optimal up to constant factors and achieve probability of capture in the first $n$ steps of order $1/\log n$. We show these strategies yield a Kakeya set consisting of $4n$ triangles with minimal area, (up to constant), namely $\Theta(1/\log n)$. As far as we know,
this is the first non-iterative construction of a boundary-optimal Kakeya set.
Considering the continuum analog of the game yields a construction of a random Kakeya set from two independent standard Brownian motions $\{B(s): s \ge 0\}$ and $\{W(s): s \ge 0\}$. Let $\tau_t:=\min\{s \ge 0: B(s)=t\}$. Then $X_t=W(\tau_t)$  is a Cauchy process,
and $K:=\{(a,X_t+at) : a,t \in [0,1]\}$ is a Kakeya set of zero area. The area of the $\epsilon$-neighborhood of $K$ is as small as possible, i.e., almost surely of order $\Theta(1/|\log \epsilon|)$. 
\newline
MSC 2010 \emph{subject classifications.}
Primary 49N75; 
secondary 05C57, 60G50. 
\end{abstract}

\begin{figure}
\begin{center}
\epsfig{file=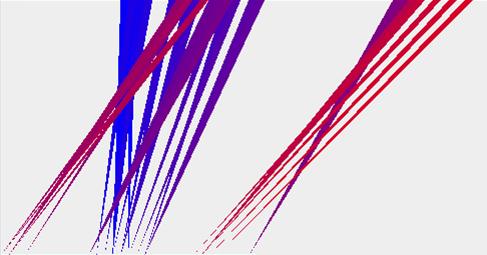, width=9cm}
\caption{\label{fig:sim-triangles} The hunter and rabbit construct a Kakeya set (see Section~\ref{sec:kakeya})
}
\end{center}
\end{figure}

\section{Introduction}

A subset $S$ of $\R^2$ is called a {\bf Kakeya} set if for every point $P$ on the unit circle in $\R^2$ there is a translation of
the line segment $(0,P)$ that is contained in $S$. A deterministic construction of a Kakeya set of zero area was first given by Besicovitch~\cite{Bes1928} in 1928.
Perron~\cite{Perron} in 1928 published a new proof of that theorem. Schoenberg~\cite{Schoenberg} constructed a Kakeya set  consisting of $4n$ triangles of area $\Theta(1/\log n)$; his construction is explained in~\cite{Besicovitch}. A similar construction was given by Keich~\cite{Keich}, who also proved that any Kakeya set consisting of $4n$ triangles cannot have area of smaller order, see~\cite[Theorem~2]{Keich}. 

In the present work we construct a new class of optimal Kakeya sets using optimal strategies in a certain game of pursuit and evasion.


\begin{definition}\rm{

Let $\G_n$ be the following two-player zero-sum game.  At every time step each player occupies a vertex of the cycle $\Z_n$.
At time $0$ the hunter and rabbit choose arbitrary initial positions. At each subsequent  step the hunter may move to an adjacent vertex or stay where she is; simultaneously the rabbit may stay where he is or move to any vertex on the cycle.
Neither player can see the other's position.  The game ends at ``capture time'' when the two players occupy the same vertex at the same time.
The hunter's goal is to minimize expected capture time; the rabbit's goal is to maximize it.
}
\end{definition}

\begin{theorem}{{\bf \cite{Rabbit}}}\label{thm:cauchyrabbit}
There exists a randomized strategy for the rabbit in the game $\G_n$ so that against any strategy for the hunter, the capture time $\tau$ satisfies
\[
\E{\tau} \geq c_1 n \log n,
\]
where $c_1$ is a fixed positive constant.
\end{theorem}

The rabbit's strategy is based on a discretized Cauchy walk; in Section~\ref{sec:lowerbound} we give a new proof of this theorem that relies on symmetry properties of simple random walk in two dimensions.

The bound $n\log n$ given in Theorem~\ref{thm:cauchyrabbit} is sharp, in the following sense:

\begin{theorem}{\bf \cite{Rabbit}}\label{thm:upperbound}
There exists a randomized strategy for the hunter in the game $\G_n$ so that against any strategy for the rabbit, the capture time $\tau$ satisfies
\[
\E{\tau} \leq c_2n\log n,
\]
where $c_2$ is a positive constant.
\end{theorem}


In Section~\ref{sec:upperbound} we give a self-contained proof of this theorem that will be useful in making the connection to Kakeya sets. 
Combining the randomized strategies of the hunter and the rabbit of Theorems~\ref{thm:cauchyrabbit} and~\ref{thm:upperbound} we prove the following theorem in Section~\ref{sec:kakeya}.

\begin{theorem}\label{thm:Kakeya}
For all $n$ there exists a Kakeya set of area at most of order $1/\log n$, which is the union of $4n$ triangles.
\end{theorem}


A central open problem regarding Kakeya sets is to understand their
Hausdorff dimension and Minkowski dimension. Davis~\cite{Davies}
proved that every Kakeya set in $\R^2$ has full Hausdorff
and Minkowski dimensions, i.e., dimension~$2$. In dimension $d>2$, it
is an open question whether every Kakeya set has full Hausdorff or
Minkowski dimension.

Minkowski dimension of a set $K \subset \R^n$ is closely
related to the area of its $\epsilon$-neighbourhood, denoted by
$K(\epsilon)$. Therefore it is natural to examine the sets
$K(\epsilon)$. The first  question that arises is the
following. What is the minimal area of $K(\epsilon)$? The answer
is known to be $\Theta(1/|\log \epsilon|)$.

\begin{proposition}\label{pro:min-kakeya}
For every Kakeya set $K \subset \R^2$ and every sufficiently
small $\epsilon >0$ we have $\vol(K(\epsilon))\geq 1/(3|\log
\epsilon|)$.
\end{proposition}

This proposition was proved by Keich~\cite{Keich} using a maximal inequality
from Bourgain's paper~\cite{Bourgain}. In Section~\ref{sec:Cauchyproc} we give an alternative
elementary proof, which can also be found in Ben Green's lecture
notes~\cite{BenGreen}.

Proposition~\ref{pro:min-kakeya} motivates the following definition.
A Kakeya set $K$ will be called \textit{optimal} if it
holds that $\vol(K(\epsilon))=O(1/|\log \epsilon|)$ as $\epsilon \to 0$. Note that every
optimal Kakeya set must have zero area.

Construction of optimal Kakeya sets is known in the literature, see
Keich~\cite{Keich}, but the construction is quite involved. Next, we describe  
a continuum analog of the Kakeya construction of Theorem~\ref{thm:Kakeya}. This simple probabilistic construction  almost surely yields an
optimal Kakeya set.

Let $\{B(s): s \ge 0\}$ and $\{W(s): s \ge 0\}$ be two independent standard Brownian motions, and let $\tau_t:=\min\{s \ge 0: B(s)=t\}$. Then $X_t=W(\tau_t)$  is a Cauchy process, i.e., a L\'evy process where the increments $X_{s+t}-X_s$ have the same law as $tX_1$, and $X_1$ has the Cauchy density ${\pi (1+x^2)}^{-1}$. See, e.g.,~\cite{Bertoin} or~\cite[Theorem~2.37]{BM}.
 
\begin{theorem}\label{thm:Cauchyproc} 
Let $\{X_t: t\geq 0\}$ be a Cauchy process and let
$\Lambda:=\{(a,X_t+at):a,t \in [0,1]\}$. Then the union
$\cup_{k=0}^{3} e^{i\pi k/4}\Lambda$ of four rotated copies of
$\Lambda$ is almost surely an optimal Kakeya set, i.e.\ there exist positive constants $c_1,c_2$ such that as $\epsilon \to 0$ we have
\[
\frac{c_1}{|\log \epsilon| } \leq \vol(\Lambda(\epsilon)) \leq \frac{c_2}{|\log \epsilon|} \ \text{ a.s.}
\]
\end{theorem}

This theorem is proved in Section~\ref{sec:Cauchyproc}.

%
%
%
%
%

\section{Probability of collision}
\label{sec:collision}

In this section we define a win-lose variant of $\G_n$ called $\G'_n$, in which only $n$ moves are made
and the hunter wins if she captures the rabbit.
\newline
Let
\[
\H = \{(H_t)_{t=0}^{n-1}: H_t \in \Z_n, |H_{t+1} - H_t | \leq 1 \} \ \text{ and } \ \RR = \{(R_t)_{t=0}^{n-1}: R_t \in \Z_n \}.
\]
Then the sets of mixed strategies $\Delta_h$ for the hunter and $\Delta_r$ for the rabbit are given by
\[
\Delta_h = \left\{x \in \R^{|\H|}: \forall f \in \H, \ x_{f} \geq 0, \  \sum_{f \in \H} x_{f} =1  \right\} \text{ and }
\Delta_r = \left\{y \in \R^{|\RR|}: \forall g \in \RR, \ y_g \geq 0, \ \sum_{g \in \RR} y_g =1 \right\}.
\]
The hunter wins $\G'_n$ if she captures the rabbit.
The pay off matrix $M = (m_{fg})_{f,g}$, where $f \in \H$ and $g \in \RR$, is given by
\[
m_{fg} = \1(\exists \ \ell \leq n-1: \ f(\ell) = g(\ell) ).
\]
When the hunter and the rabbit use the mixed strategies $x$ and $y$ respectively, then
\[
x^T M y = \prstart{\tau <n}{xy},
\]
where $\tau$ is the capture time.

By the minimax theorem we have
\[
\max_{x \in \Delta_h} \min_{y \in \Delta_r} x^T M y = \min_{y \in \Delta_r} \max_{x \in \Delta_h} x^T M y = \val(\G'_n),
\]
where $\val(\G'_n)$ stands for the value of the game (to the hunter). Thus, there exists a randomized strategy for the hunter
so that against every strategy of the rabbit the probability that they collide in the first $n$ steps is $\val(\G'_n)$;
and there exists a randomized strategy for the rabbit, so that against every strategy of the hunter, the probability they collide
is $\val(\G'_n)$.

\begin{remark}\rm{
In Sections~\ref{sec:lowerbound} and~\ref{sec:upperbound} we give randomized strategies for the hunter and the rabbit that
achieve $\val(\G'_n)$ up to multiplicative constants. In particular, in Propositions~\ref{pro:probupperbound} and~\ref{pro:upperbound} we show there are positive constants $c_3$ and $c_4$ such that
\begin{align}\label{eq:value}
\frac{c_3}{\log n} \leq \val(\G'_n)\leq \frac{c_4}{\log n}.
\end{align}
}
\end{remark}

\begin{lemma}\label{lem:strategy}

Let $\til{x} \in \Delta_h$ be a randomized hunter strategy in the game $\G'_n$
satisfying $\min_{y \in \Delta_r} \til{x}^TMy = p_n$. Then there exists a randomized hunter strategy in the game $\G_n$
so that against any rabbit strategy, the capture time $\tau$ satisfies
\[
\E{\tau} \leq \frac{2n}{p_n}.
\]
Let $\til{y}\in \Delta_r$ be a randomized rabbit strategy in the game $\G'_n$ satisfying $\max_{x \in \Delta_h} x^TM\til{y} = q_n$. Then
there exists a randomized rabbit strategy in the game $\G_n$ so that against any hunter strategy, the capture time $\tau$ satisfies
\[
\frac{n}{q_n} \leq \E{\tau}.
\]
\end{lemma}

\begin{proof}[{\bf Proof}]

We divide time into rounds of length $n$. In rounds $1,3,5,\ldots$ the hunter employs independent copies of the randomized strategy
$\til{x}$ and she uses the even rounds to move to the proper starting points.

This way we get a new hunter strategy $\xi$ so that against any rabbit strategy $\eta$ in $\G_n$
\[
\prstart{\tau<2n} {\xi \eta}\geq \prstart{\tau<n}{\til{x}\eta'} = \til{x}^TMy\geq p_n,
\]
where $\eta'$ is the restriction of the strategy $\eta$ in the first $n$ steps.
Therefore, by bounding the capture time $\tau$ by $2n$ times a geometric random variable of success probability $p_n$, we get
$\E{\tau} \leq \frac{2n}{p_n}$.

For the lower bound we look at the process in rounds of length $n$. In each round the rabbit employs an independent copy
of the randomized strategy $\til{y}$. Thus the capture time stochastically dominates $n$ times a geometric random variable of
parameter $q_n$, and hence $\E{\tau} \geq \frac{n}{q_n}$.
\end{proof}

\section{The rabbit's strategy}\label{sec:lowerbound}

In this section we give the proof of Theorem~\ref{thm:cauchyrabbit}. We start with a standard
result for random walks in $2$ dimensions and include its proof for the sake of completeness.

%

\begin{figure}
\begin{center}
\subfigure[Escaping the $2i\times 2i$ square]{
\epsfig{file=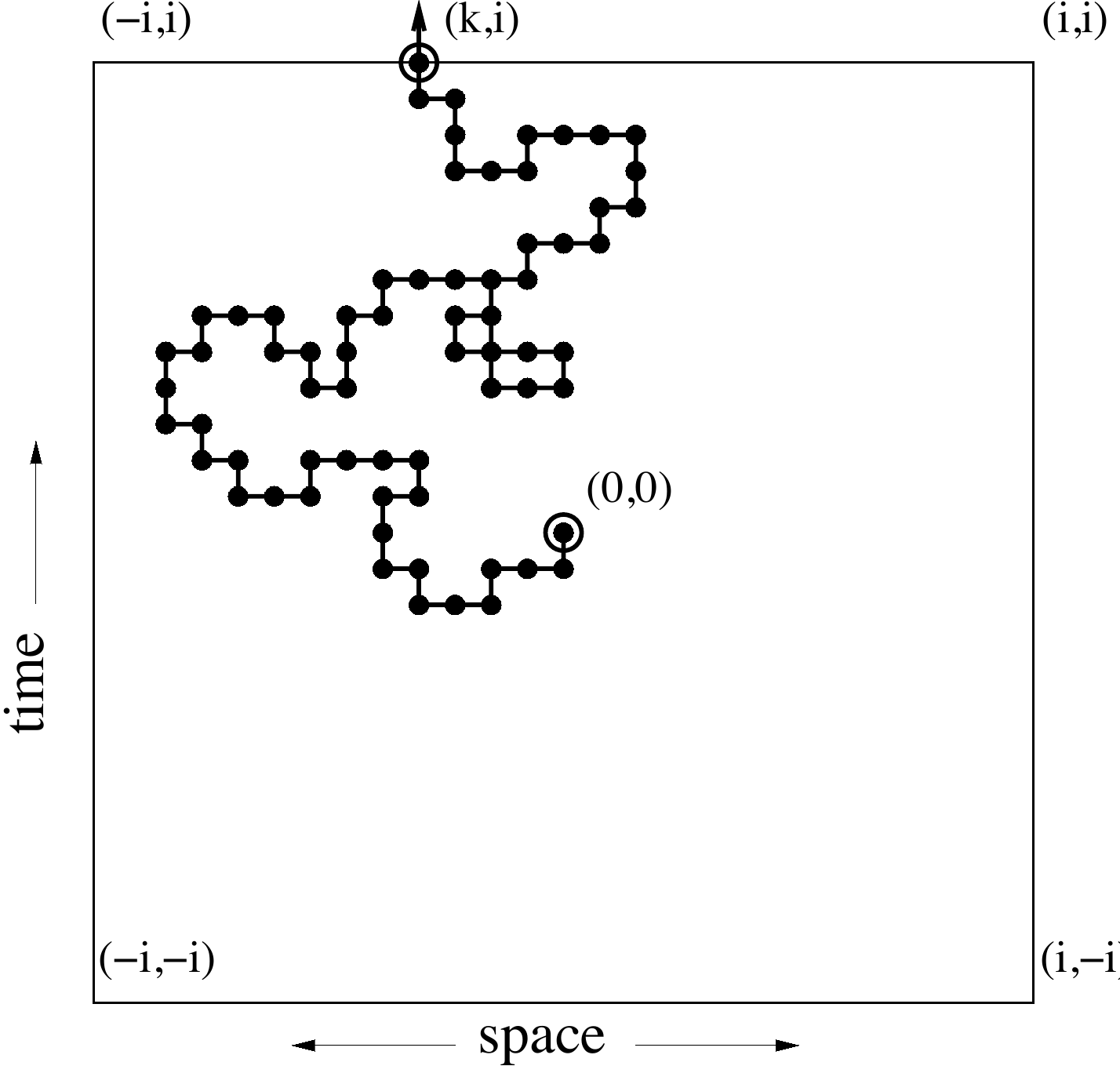,height=7cm}
}
\subfigure[Escaping the $2k\times 2k$ square]{
\epsfig{file=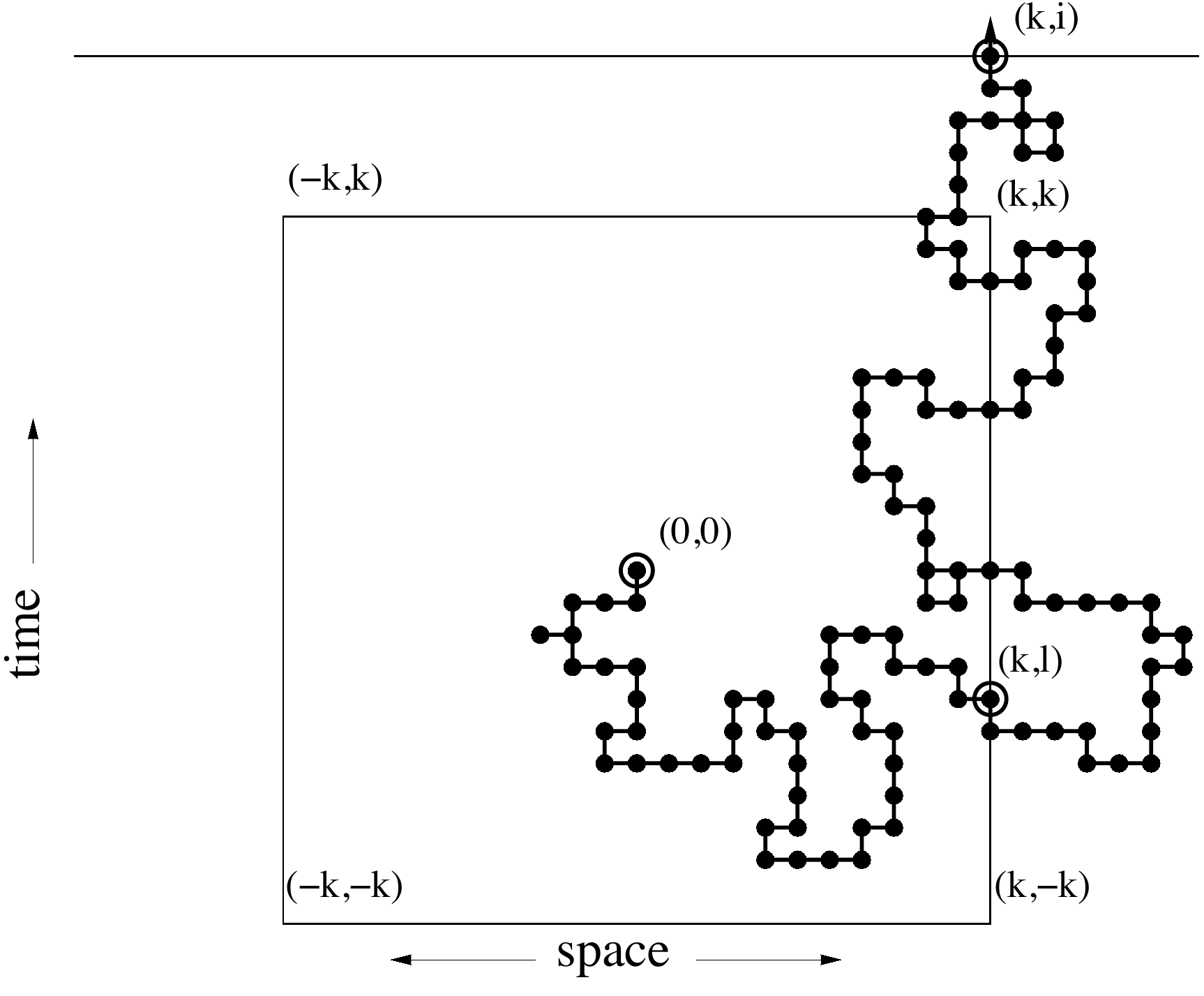,height=7cm}
}
\caption{\label{fig:hittingtimes} Hitting times}
\end{center}
\end{figure}

\begin{lemma}\label{lem:randomwalk}
Let $Z= (X,Y)$ be a simple random walk in $\Z^2$ starting from $0$. For every $i \in \Z$ define $T_i = \inf\{ t\geq 0: Y_t = i\}$.
Then for all $k \in \{-i,\ldots, i\}$ we have
\[
\prstart{X_{T_i} = k}{0} \geq \frac{1}{96 i}.
\]
\end{lemma}

\begin{proof}[{\bf Proof}]

Notice that $X_{T_i}$ has the same distribution for both a lazy simple random walk and a non-lazy one.
So it suffices to prove the lemma in the case of a lazy simple random walk in $\Z^2$. We realize a lazy simple random walk in $\Z^2$ as follows:

Let $V$, $W$ be two independent lazy simple random walks in $\Z$. Let $\xi_1,\xi_2,\ldots$  be i.i.d.\ random variables taking values
$1$ or $2$ with equal likelihood. Now for all $k$ define
$r(k) = \sum_{i=1}^{k}\1(\xi_i=1)$ and let
\[
(X_k,Y_k) = (V_{r(k)},W_{k-r(k)}).
\]
Then it is elementary to check that $Z = (X,Y)$ is a lazy simple random walk in $\Z^2$. \newline
We first show that for all $k \in \{-i,\ldots,i\}$,
\begin{align}\label{eq:oversusk}
\prstart{X_{T_i} = 0}{0} \geq \prstart{ X_{T_i} = k}{0}.
\end{align}
Since $V$ is independent of $T_i$ and of $r(\ell)$ for all values of $\ell$, we get for all $k$
\begin{align}\label{eq:returnk}
\prstart{X_{T_i} = k}{0} & = \sum_{m,\ell} \prstart{X_m = k, T_i = m, r(m) = \ell}{0} = \sum_{m,\ell}\prstart{V_\ell = k}{0} \prstart{r(m)=\ell,T_i =m}{0}.
\end{align}
It is standard (see, for example, \cite[Lemma~12.2]{Mixingtimesbook}) that for a lazy simple random walk on $\Z$, if $P^t$ stands for the transition probability in $t$ steps
then $P^t(x,y) \leq P^t(x,x)$ for all $x$ and $y$. Applying this to $V_\ell$ and using~\eqref{eq:returnk} we obtain
\begin{align*}
\prstart{X_{T_i} = 0}{0}
\geq \sum_{m,\ell}\prstart{V_\ell = k}{0} \prstart{r(m)=\ell,T_i =m}{0} = \prstart{X_{T_i} = k}{0}
\end{align*}
and this concludes the proof of~\eqref{eq:oversusk}.

\medskip

For $k \in \Z$ we let
\[
\tau_k = \min\{t\geq 0: X_t \notin [-|k|+1,|k|-1]^2\}.
\]
Setting $A = \{Y_{\tau_i}=i\}$ we have by symmetry
\[
\prstart{X_{T_i} \in \{-i,\ldots, i\}}{0} \geq \prstart{A}{0} = \frac{1}{4}.
\]
Hence this together with~\eqref{eq:oversusk} gives that
\begin{align}\label{eq:lowerboundi}
\prstart{X_{T_i} = 0}{0} \geq \frac{1}{8i+4} \geq \frac{1}{12i}.
\end{align}
To finish the proof of the lemma it remains to show that for all $(k,i)$ with $k \in \{-i,\ldots, i\}$ we have
\begin{align}\label{eq:finalgoal}
\prstart{X_{T_i} = k}{0} \geq \frac{1}{96 i}.
\end{align}
For any $k \in \{-i,\ldots, i\}$ we have
\begin{align}\label{eq:sum}
\nonumber
\prstart{X_{T_i} =k}{0} &\geq \prstart{X_{\tau_k} = k, X_{T_i} = k}{0} = \sum_{\ell = -|k|,\ldots, |k|} \prstart{X_{\tau_k} = k,Y_{\tau_k} = \ell, X_{T_i} =k}{0}\\
& = \sum_{\ell} \prcond{X_{T_i} = k}{X_{\tau_k} =k,Y_{\tau_k} =\ell}{0} \prstart{X_{\tau_k} = k,Y_{\tau_k} =\ell}{0}.
\end{align}

Notice that by the strong Markov property, translation invariance and applying~\eqref{eq:lowerboundi} to $i-\ell$ we get
\[
\prcond{X_{T_i} = k}{X_{\tau_k} =k,Y_{\tau_k} =\ell}{0}  = \prstart{X_{T_{i-\ell}} = k}{(k,\ell)} = \prstart{X_{T_{i-\ell}} =0}{0} \geq \frac{1}{12(i-\ell)} \wedge 1\geq \frac{1}{24 i},
\]
since $-k <\ell < k$ and $k \leq i$.
Therefore, plugging this into~\eqref{eq:sum}, we obtain
\[
\prstart{X_{T_i} =k}{0} \geq \frac{1}{24 i}\sum_{\ell = -|k|,\ldots,|k|} \prstart{X_{\tau_k}=k,Y_{\tau_k} =\ell}{0} = \frac{1}{24 i}\prstart{X_{\tau_k} = k}{0} \geq \frac{1}{96 i}
\]
since by symmetry we have $\prstart{X_{\tau_k} = k}{0} \geq  1/4$.
This concludes the proof of the lemma.
\end{proof}

\begin{proposition}\label{pro:probupperbound}
There exists a randomized rabbit strategy in the game $\G_n$ so that against any hunter strategy the capture time $\tau$ satisfies
\[
\pr{\tau< n} \leq \frac{c}{\log n},
\]
where $c$ is a universal constant.
\end{proposition}

\begin{proof}[{\bf Proof}]
It suffices to prove the upper bound for a pure strategy of the hunter, i.e.\ a fixed path $(H_i)_{i<n}$.

Let $U$ be uniformly distributed on $\{0,\ldots,n-1\}$ and let $Z = (X,Y)$ be an independent simple random walk in $\Z^2$.
We define a sequence of stopping times as follows: $T_0= 0$ and inductively for $k\geq 0$,
\[
T_{k+1} = \inf\{t\geq T_k: Y_t = k+1\}.
\]
By recurrence of the two-dimensional random walk, for all $k$ we have $T_k <\infty$ a.s.

Define the position of the rabbit at time $0$ to be $R_0 = U$ and
$R_k = (X_{T_k} + U)\bmod n$ at time $k$.
Define $K_n$ to be the total number of collisions in the first $n$ steps, i.e.\ $K_n =\sum_{i=0}^{n-1} \1(H_i = R_i)$.
Since $\{\tau<n\} = \{K_n>0\}$, it suffices to show that for a positive constant $c$,
\begin{align}\label{eq:firstgoal}
\pr{K_n >0}  \leq \frac{c}{\log n}.
\end{align}
For the rest of the proof we extend the sequence $(H_i)_{i<n}$ by defining $H_{i+n} = H_n$ for all $i<n$.
Then we have
\begin{align}\label{eq:obvious}
\pr{K_n >0} \leq \frac{\E{K_{2n}}}{\econd{K_{2n}}{K_n>0}}.
\end{align}
In order to prove~\eqref{eq:firstgoal} we will bound the numerator and denominator separately. \newline
Since $U$ is uniform on $\{0,\ldots,n-1\}$ and
$X$ is independent of $U$, it follows that $R_i$ is uniform on $\{0,\ldots,n-1\}$ for every $i$.
Using that and the fact that the hunter and the rabbit move independently, we deduce
\begin{align}\label{eq:expectation}
\E{K_{2n}} = \sum_{i=0}^{2n-1} \pr{R_i =H_i } = \sum_{i=0}^{2n-1}\frac{1}{n} = 2.
\end{align}
For the term $\econd{K_{2n}}{K_n>0}$ we have
\begin{align}\label{eq:coprabbitmarkov}
\econd{K_{2n}}{K_n>0} &= \sum_{k=0}^{n-1} \econd{\sum_{i=k}^{2n-1} \1(R_i=H_i)}{\tau=k}   \frac{\pr{\tau=k}}{\pr{\tau<n}}.
\end{align}
Define $\til{R}_i = (R_{k+i} - R_k)\bmod n$ and $\til{H}_i = (H_{i+k} - H_k)\bmod n$. By the definition of the process $R$ it follows that $\til{R}$ has the same law as the process $R$.
By the Markov property of the process $(R_i)$ we get
\begin{align}\label{eq:condprobcoprob}
\econd{\sum_{i=k}^{2n-1} \1(R_i=H_i)}{\tau=k}  \geq \E{\sum_{i=0}^{n-1} \1(\til{R}_i = \til{H}_{i})}
 = 1 +\sum_{i=1}^{n-1} \prstart{{R}_i = \til{H}_i}{0}.
\end{align}
For all $i\geq 1$ and all $\ell \in \{-i,\ldots,i\}$, since $R_i = X_{T_i}\bmod n$, using Lemma~\ref{lem:randomwalk} we get
\begin{align}\label{eq:rabbit}
\prstart{R_{i} = \ell \bmod{n}}{0} \geq \prstart{X_{T_i} = \ell}{0} \geq  \frac{1}{64i}.
\end{align}
For all $i$ we have $\til{H}_i \in \{-i \bmod n,\ldots,i \bmod n\}$, since $\til{H}_0=0$. Using~\eqref{eq:rabbit} we get that for all $i\geq 1$
\[
\prstart{R_i =\til{H}_i}{0} \geq \frac{1}{96i}.
\]
The above inequality together with~\eqref{eq:coprabbitmarkov} and~\eqref{eq:condprobcoprob} yields
\[
\econd{K_{2n}}{K_n>0} \geq 1 + \sum_{i=1}^{n-1}\frac{1}{96 i} \geq c_1 \log n,
\]
where $c_1$ is a positive constant.
Thus~\eqref{eq:obvious}, the above inequality and~\eqref{eq:expectation}  conclude the proof of~\eqref{eq:firstgoal}.
\end{proof}

\begin{remark}\rm{
We refer to the strategy used by the rabbit in the proof above as the Cauchy strategy, because it is the discrete analogue
of the hitting distribution of planar Brownian motion on a line at distance $1$ from the starting point, which is the Cauchy distribution.
}
\end{remark}

\begin{proof}[{\bf Proof of Theorem~\ref{thm:cauchyrabbit}}]

The proof of the theorem follows by combining Lemma~\ref{lem:strategy} and Proposition~\ref{pro:probupperbound}.
\end{proof}

\section{ The hunter's strategy}\label{sec:upperbound}
In this section we give the proof of Theorem~\ref{thm:upperbound} by constructing a randomized strategy for the hunter. Before doing so, it is perhaps useful to consider the following natural strategy for the hunter: at time $0$ she chooses a random location and a random direction. Subsequently
at each time $t$ she continues in the same direction she has been walking with probability $(n-2)/n$, stops for one move and then continues in the same direction with probability $1/n$, and reverses direction with probability $1/n$.
We call this the ``zigzag'' strategy. 
We can prove that the zigzag strategy achieves expected capture time of order $n^{3/2}$ against any rabbit strategy. 
The following counter-strategy of the rabbit yields expected capture time of $n^{3/2}$ against the zigzag strategy: he starts at random, walks for $\sqrt{n}$ steps to the right in unit steps, then jumps to $2\sqrt{n}$ to the left and repeats. 

To achieve minimal expected capture time (up to a constant) our hunter moves not only in a random direction but at a random rate.

\begin{figure}
\begin{center}
\subfigure[Typical hunter path, zigzag strategy]{\label{fig:diagonals}
\epsfig{file=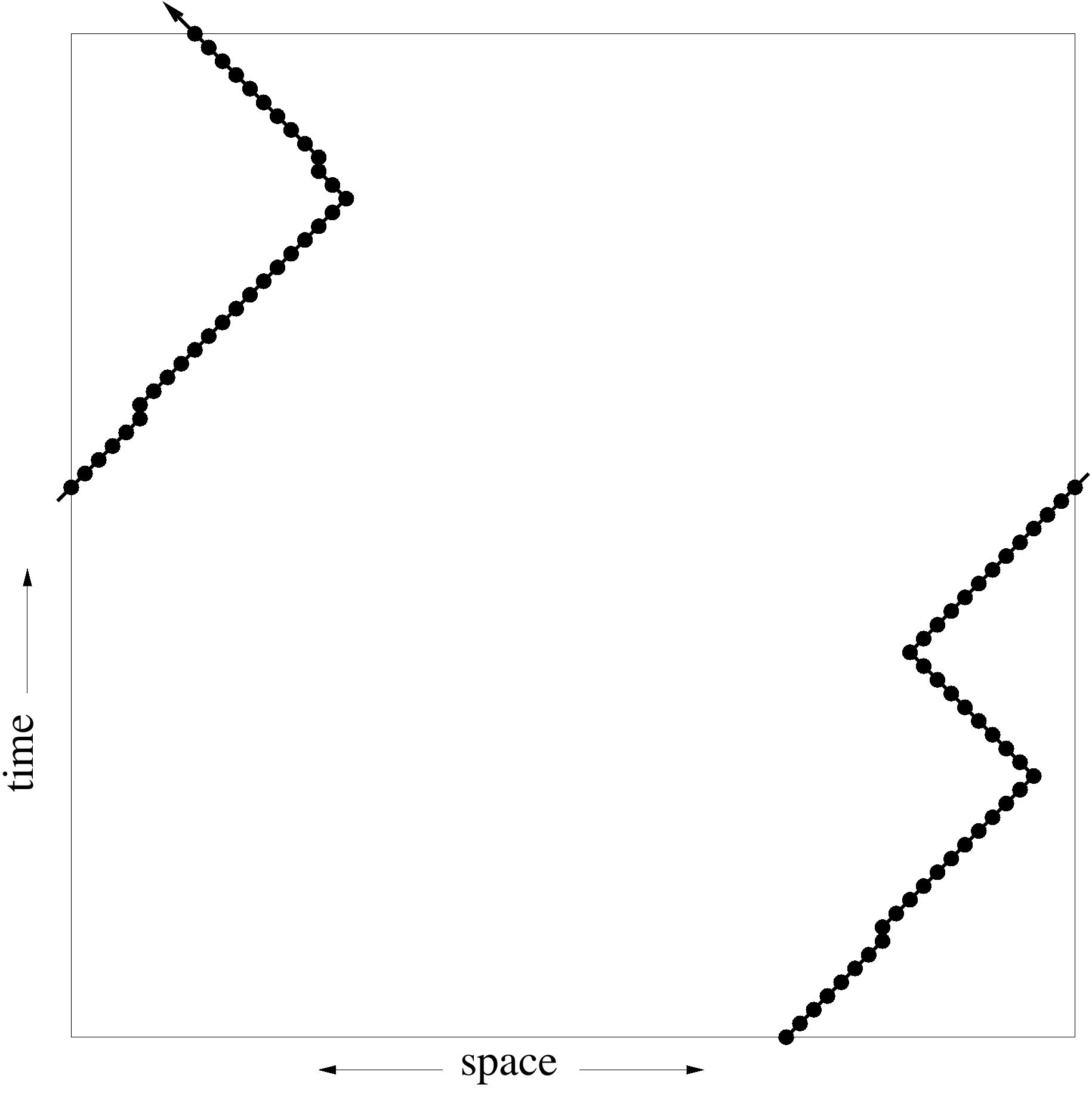,height=7cm}
}
\subfigure[Typical rabbit path, counter to zigzag strategy]{\label{fig:diagonalrabbit}
\epsfig{file=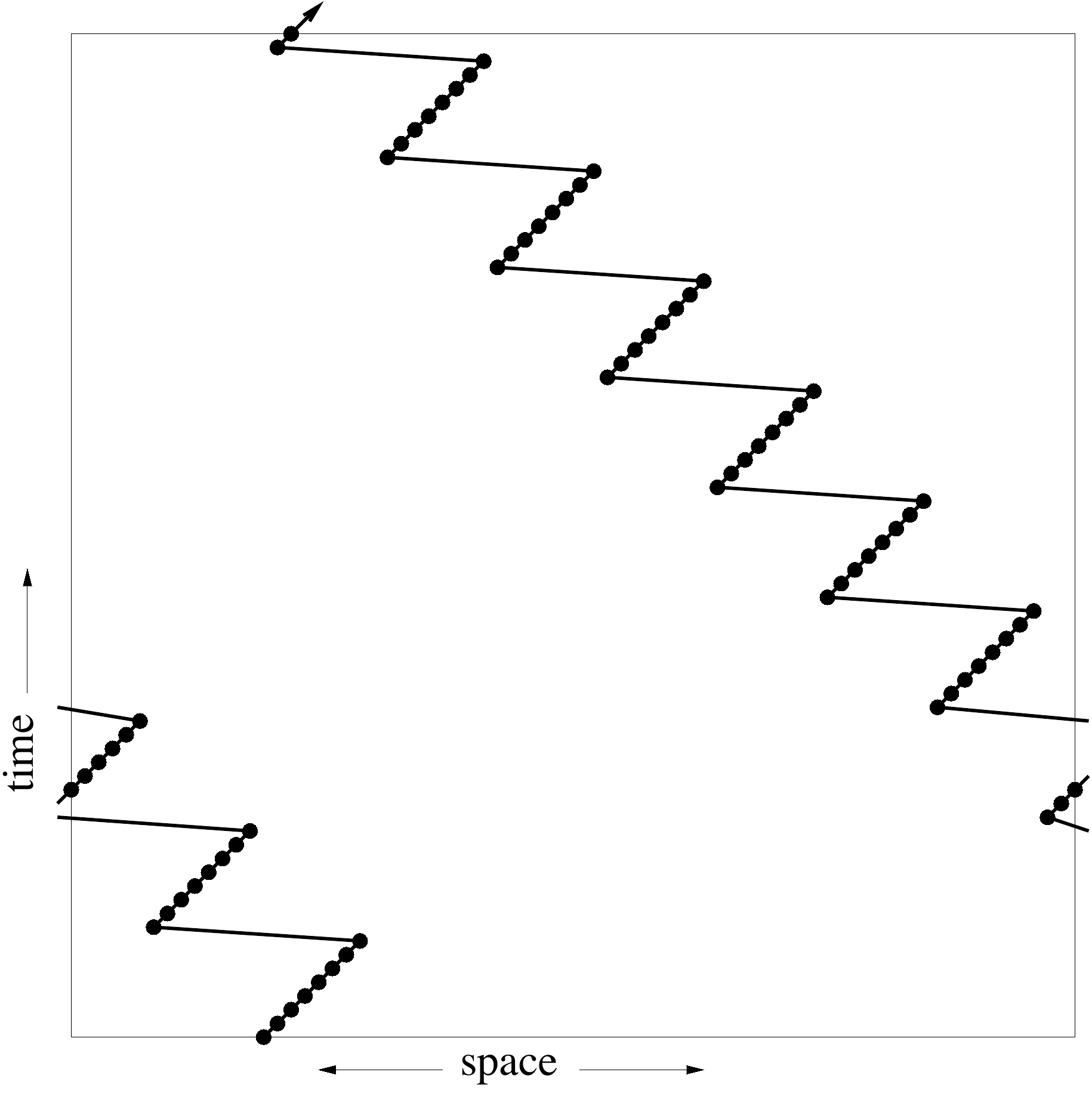,height=7cm}
}
\caption{Typical paths}
\end{center}
\end{figure}

\begin{proposition}\label{pro:upperbound}
There exists a randomized hunter strategy in the game $\G'_n$ so that against any rabbit strategy the capture time $\tau$ satisfies
\[
\pr{\tau <n } \geq \frac{c'}{\log n},
\]
where $c'$ is a universal positive constant.
\end{proposition}

\begin{proof}[{\bf Proof}]

Let $R_\ell$ be the location of the rabbit on the cycle at time $\ell$, i.e.\ $R_\ell \in \{0,\ldots, n-1\}$.
We now describe the strategy of the hunter. Let $a,b$ be independent random variables uniformly distributed on $[0,1]$.
We define the location of the hunter at time $\ell$ to be $H_\ell = \lceil an +b\ell\rceil \bmod n$.

We again let $K_n$ denote the number of collisions before time $n$, i.e. $K_n = \sum_{i=0}^{n-1}\1(R_i=H_i)$. Then by the second moment method we have
\[
\pr{K_n>0} \geq \frac{(\E{K_n})^2}{\E{K_n^2}}.
\]
We now compute the first and second moments of $K_n$. For that, let $I_\ell$ denote the event that there is a collision at time $\ell$,
i.e.\ $I_\ell = \{H_\ell = R_\ell\}$. We first calculate $\pr{I_\ell}$. We have
\[
I_\ell = \{\lceil an +b\ell\rceil =R_\ell \} \cup \{\lceil an +b\ell\rceil -n = R_\ell \},
\]
which gives that
\[
I_\ell = \{ R_\ell -1 <  an +b\ell \leq R_\ell\} \cup \{ R_\ell + n -1< an + b\ell \leq R_\ell + n\}.
\]
Hence $\pr{I_\ell} = 1/n$ and
\[
\E{K_n} = \sum_{\ell=0}^{n-1} \pr{I_\ell} = 1.
\]
Let $j>0$, then it is easy to check that $\pr{I_\ell\cap I_{\ell+j}}\leq \frac{c}{jn}$ for a positive constant $c$.
Therefore
\begin{align*}
\E{K_n^2} &= \E{\left(\sum_{\ell = 0}^{n-1} I_\ell\right)^2}  = \E{K_n} + \sum_{\ell \neq m} \E{I_\ell \cap I_m}
= 1 + 2 \sum_{\ell= 0}^{n-1} \sum_{j=1}^{n-\ell-1} \pr{I_\ell \cap I_{\ell+j}} \\
&\leq 1+2 \sum_{\ell = 0}^{n-1} \sum_{j=1}^{n} \frac{c}{jn} \leq c' \log n,
\end{align*}
for a positive constant $c'$.
This way we get
\[
\pr{\tau< n }  = \pr{K_n >0} \geq \frac{c_1}{\log n}
\]
and this finishes the proof of the proposition.
\end{proof}

\begin{proof}[{\bf Proof of Theorem~\ref{thm:upperbound}}]
The proof of the theorem follows from Lemma~\ref{lem:strategy} and Proposition~\ref{pro:upperbound}.
\end{proof}

\section{The Kakeya connection}\label{sec:kakeya}

In this section we prove Theorem~\ref{thm:Kakeya}. We start by
showing how to get a Kakeya set given a strategy of the rabbit with probability at most $p_n$ of collision against any strategy of the hunter.

\begin{proposition}\label{pro:kakeya}
Given a strategy $\til{y}$ of the rabbit which ensures capture probability at most $p_n$ against any hunter strategy,
there is a Kakeya set of area at most $8p_n$ which is the union of $4n$ triangles.
\end{proposition}

\begin{figure}
\begin{center}
\epsfig{file=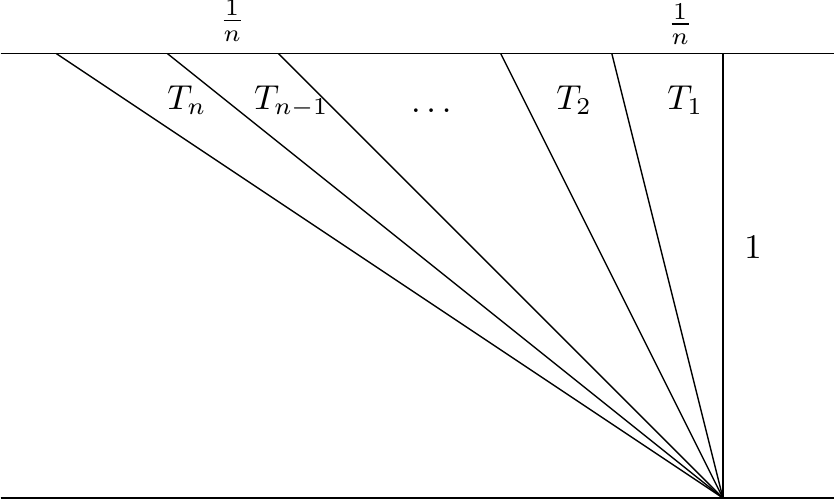, width=7 cm}
\caption{\label{fig:manytriangles} Triangles}
\end{center}
\end{figure}

\begin{proof}[{\bf Proof}]

Recall the definition of the set $\H$ of the allowed hunter paths.

First we slightly change the game and enlarge the set of allowed paths for the hunter, to include all functions $f:[0,n) \to  [0,n)$ that are $1$-Lipschitz.
Then we say that there is a collision in $[0,n)$ if for some $m \leq n-1$ there exists $t \in [m,m+1)$ such that $f(t) = R_m$.
We first show that if $f$ is $1$-Lipschitz, then
\begin{align}\label{eq:goalvalue}
\min_{y \in \Delta_r}\prstart{\text{\rm collision in } [0,n)}{f y} \leq p_n,
\end{align}
where $f$ stands for $\delta_f$ with a slight abuse of notation.
\newline
In order to prove~\eqref{eq:goalvalue}, for every $f$ that is $1$-Lipschitz we will construct a path $h \in \H$ such that for all $y \in \Delta_r$
\begin{align}\label{eq:goalcollision}
\prstart{\text{\rm collision in } [0,n)}{f y} \leq
\prstart{\tau<n }{{h} y}.
\end{align}
We define $h(m)$ for every $m \in \{0,\ldots,n-1\}$. By the $1$-Lipschitz property of $f$, note that the image $f([m,m+1))$ can contain at most one integer.
If there exists $k\in \Z_n$ such that $k \in f([m,m+1))$, then we set $h(m) = k$. If there is no such integer $k$, then we set $h(m) = \lfloor f(m) \rfloor$.
The $1$-Lipschitz property then gives that $h\in \H$.
Since the rabbit only jumps on $\Z_n$, the function $h$ constructed this way satisfies~\eqref{eq:goalcollision}.

Applying~\eqref{eq:goalcollision} to the strategy $\til{y}$ of the rabbit and using the assumption gives that for all $f$ that are $1$-Lipschitz
\begin{align}\label{eq:pnbound}
\prstart{\text{\rm collision in } [0,n)}{f \til{y}} \leq
p_n.
\end{align}

Next we consider the hunter strategy in which she chooses a linear function $f_{a,b}(t) = (an+tb) \bmod n$, where $a,b$ are independent and
uniformly distributed on the unit interval $[0,1]$.
Suppose that during the time segment $[m, m + 1)$ the rabbit is located at position $z_m$. Then the set of values $(a, b)$ such that $z_m \in f_{a,b}([m,m+1))$ is
$T(m,z_m) = T_{\ell}(m,z_m) \cup T_r(m,z_m)$, where
\begin{align*}
&T_\ell(m,z_m) = \{an+bm \leq z_m < an+b(m+1)\}\cap [0,1]^2 \ \text{ and } \\
& T_r(m,z_m)= \{an+bm -n\leq z_m <an+b(m+1)-n\}\cap [0,1]^2,
\end{align*}
as illustrated in Figure~\ref{fig:twopieces}.
If the rabbit chooses the sequence of locations $(z_k)_{k=0}^{n-1}$, then he will be caught by the hunter using the strategy above with
probability $A(z)$ which is the area of $\cup_m T(m,z_m)$. Therefore the objective of the rabbit is to minimize the area of $A(z)$.
￼￼
We have
\[
A(R) = \prcond{\text{collision in } [0,n)}{(R_m)}{},
\]
and hence since from~\eqref{eq:pnbound} we have $\prstart{\text{collision in } [0,n)}{f_{a,b} \til{y}} \leq \int_0^1\int_0^1 p_n \,da db = p_n$,
we deduce that for the strategy $\til{y}$
\[
\E{A(R)} \leq p_n.
\]

\begin{figure}
\begin{center}
\epsfig{file=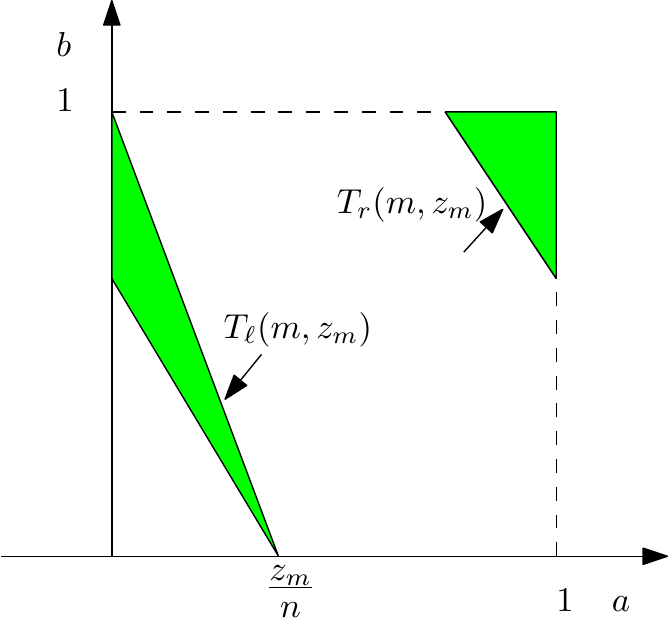, width=7 cm}
\caption{\label{fig:twopieces} $T(m,z_m) = T_\ell(m,z_m) \cup T_r(m,z_m)$}
\end{center}
\end{figure}

Now we slightly change the sets $T(m,z_m)$, since they could consist of two disjoint triangles as illustrated in Figure~\ref{fig:twopieces}.
So if we write $T'(m,z_m) = T_\ell(m,z_m) \cup (T_r(m,z_m) - (1,0))$, then it is easy to see that $T'(m,z_m)$ is always a triangle.

Hence taking the union  $\cup_m T'(m,z_m)$ gives a union of $n$ triangles with
\[
\textrm{Area}(\cup_m T'(m,z_m)) \leq 2 \textrm{Area}(\cup_m T(m,z_m)),
\]
and hence $\E{\text{Area}(\cup_m T'(m,z_m))} \leq 2p_n$.

The triangles $T_i$ in Figure~\ref{fig:manytriangles} contain unit segments in all directions that have an angle in $[0,\pi/4]$ with the vertical axis.
Since the triangles $T'(m,z_m)$ are obtained from the triangles $T_i$ by horizontal translation, the union $\cup_mT'(m,z_m)$ also contains a unit segment
in all these directions. Hence if we take $4$ copies of this construction suitably rotated obtaining $4n$ triangles gives a Kakeya set with area at most $8p_n$.
\end{proof}

\begin{proof}[{\bf Proof of Theorem~\ref{thm:Kakeya}}]

If the rabbit uses the Cauchy strategy, then by Proposition~\ref{pro:probupperbound} we get that the probability of collision in $n$ steps against any strategy
of the hunter is at most $p_n = c/\log n$. Now we can apply Proposition~\ref{pro:kakeya} to get $4n$ triangles of area at most $8p_n = 8c/\log n$.
For a sample of this random construction with $n=$ see Figure~\ref{fig:sim-triangles}.
\end{proof}

\section{Kakeya sets from the Cauchy process}\label{sec:Cauchyproc}

Our goal in this section is to prove Theorem~\ref{thm:Cauchyproc}. We first recall some notation.

Let $(X_t)$ be a Cauchy process, i.e., $X$ is a stable process with values in $\R$  and the density of $X_1$ is given by $(\pi(1+x^2))^{-1}$. Let $\F_t = \sigma(X_s, s\leq t)$ be its natural filtration and let $\til{\F}_t = \cap_{n} \F_{t+1/n}$. Then $(\til{\F}_t)$ is right continuous and $X$ is adapted to $(\til{\F})$.

For any set $A$ we denote by $A(\epsilon)$ the $\epsilon$-neighbourhood of $A$, i.e.\ $A(\epsilon)  = A+\B(0,\epsilon)$. 

Let $F$ be a subset of $[0,1]$ and $\delta>0$. For $a\in [0,1]$ we define
\[
V_a(F,\delta) = \cup_{s\in F} \B(X_s+as, \delta).
\]

Recall the definition $\Lambda = \{ (a,X_t+at): 0\leq a\leq 1, 0\leq t\leq 1\}$.


\begin{lemma}\label{cl:voltr}
Let $M>0$ be a constant, $t>r$ and $I=[u,u+t]$ be a subinterval of $[0,1]$. Then there exists a constant $c=c(M)$ so that for all $a \in [-M,M]$
\[
\E{\vol(V_a(I,r) )} \leq \frac{ct}{\log(t/r)} + 2r.
\]
\end{lemma}

\begin{proof}[{\bf Proof}]

%

By translation invariance of Lebesgue measure and the stationarity of $X$, it suffices to prove the lemma in the case when $I=[0,t]$.

If  $\tau_{\B(x,r)} = \inf\{s\geq 0: X_s+as \in \B(x,r)\}$, then we can write 
\begin{align}\label{eq:volumedec}
\E{\vol(\cup_{s\leq t}\B(X_s+as,r) )} = \int_{\R} \pr{\tau_{\B(x,r)} \leq t} \,dx = 2r + \int_{\R \setminus \B(0,r)} \pr{\tau_{\B(x,r)} \leq t} \,dx.
\end{align}
For $x\notin \B(0,r)$ we define $Z_x = \int_{0}^{t} \1(X_s + as \in \B(x,r)) \,ds$ and $\til{Z}_x = \int_{0}^{2t} \1(X_s + as \in \B(x,r))$. By the c{\`a}dl{\`a}g property of $X$ we deduce that up to zero probability events we have $\{\tau_{\B(x,r)}\leq t\} = \{Z_x >0\}$. So it follows that
\begin{align}\label{eq:numden}
\pr{\tau_{\B(x,r)} \leq t} = \pr{Z_x >0} \leq \frac{\E{\til{Z}_x}}{\econd{\til{Z}_x}{Z_x>0}}.
\end{align}
For the numerator we have
\[
\E{\til{Z}_x} = \int_{0}^{2t} \int_{\B(x,r)} p_s(0,y) \,dy \,ds = \int_{0}^{2t} \int_{\B(0,r)} p_s(0,x+y) \,dy \,ds,
\]
where $p_s(0,y)$ stands for the transition density in time $s$  of the process $(X_u+au)_u$. We now drop the dependence on $\B(x,r)$ from $\tau_{\B(x,r)}$ to simplify notation. For the conditional expectation appearing in the denominator in~\eqref{eq:numden} we have

\begin{align*}
\econd{\til{Z}_x}{Z_x>0} &= \econd{\int_{\tau}^{2t} \1(X_s + as \in \B(x,r)) \, ds}{\tau\leq t} \\
&= \econd{\int_{0}^{2t-\tau} \1(X_{s+\tau} + a(s+\tau) \in \B(x,r)) \,ds}{\tau\leq t}   \\
& =  \econd{\int_{0}^{2t-\tau} \1(X_{s+\tau} -X_\tau + as+X_\tau + a\tau) \in \B(x,r)) \,ds}{\tau\leq t} 
\\
 &\geq \min_{y \in \B(x,r)}\E{\int_0^t \1(X_s + as + y\in \B(x,r)) \,ds},
\end{align*}
where in the last step we used the strong Markov property of $X$ and that $X_\tau + a\tau =y  \in \B(x,r)$.
We now bound from below the expectation appearing in the minimum above. Since we assumed that $r<t$, we have
\begin{align*}
&\E{\int_0^t \1(X_s + as + y\in \B(x,r)) \,ds} = \int_{0}^{t} \int_{\B\left(\frac{x}{s} - \frac{y}{s} - a,\frac{r}{s}\right)} \frac{1}{\pi(1+z^2)} \,dz\,ds \\
& \geq \int_{r}^{t} \frac{2r}{(1 + (M+3)^2)  \pi s} \,ds = c_1r \log \frac{t}{r}.
\end{align*}
The inequality follows from the observation that when $s\geq r$ and $y \in \B(x/s- y/s - a,r/s)$, then $|y|\leq M+3$, since $a\in [-M,M]$.
Hence we infer that for all $x$
\[
\econd{\til{Z}_x}{Z_x>0} \geq c_1 r \log(t/r).
\]
So putting all things together we have
\[
\int_{\R \setminus \B(0,r)} \pr{Z_x >0} \,dx \leq \frac{\int_{\R\setminus \B(0,r)}\int_{0}^{2t} \int_{\B(0,r)} p_s(0,x+y)\,dy\,ds\,dx}{c_1 r \log(t/r) } \leq \frac{4rt}{c_1 r \log(t/r)} = \frac{c_2 t}{\log(t/r)}
\]
and this together with~\eqref{eq:volumedec} completes the proof of the lemma.
\end{proof}

%
%
%

\begin{claim}\label{cl:stoppingtime}
Let $(\F_t)$ be a right continuous filtration and $(X_t)$ a c{\`a}dl{\`a}g adapted process taking values in $\R^d$, $d\geq 1$. Let $D$ be an open set in $\R^d$ and $F$ a subset of $[0,1]$. Then
\[
\tau = \inf\{ t\in F: X_t \in D\}
\]
is a stopping time.
\end{claim}

\begin{proof}[{\bf Proof}]
Let $F_\infty$ be a countable dense subset of $F$. Then for all $t\in [0,1]$ we deduce
\[
\{\tau <t\} = \cup_{q \in F_\infty, q<t} \{X_q \in D\},
\]
since $X$ is c{\`a}dl{\`a}g and $D$ is an open set. Hence $\{\tau<t\} \in \F_t$. Writing
\[
\{\tau\leq t\} = \bigcap_{n} \{ \tau<t+1/n\},
\] 
we get that $\{\tau\leq t\} \in \F_{t+} = \F_t$.
\end{proof}

\begin{lemma}\label{lem:yestimate}
Let $I$ be a subinterval of $[0,1]$ of length $\sqrt{\epsilon}$. Define 
$Y=\int_b^d \vol(V_a(I,2\epsilon)) \,da$, where $-2<b<d<2$.
Then there exists a constant $c$ such that for all $\epsilon>0$ sufficiently small 
\[
\E{Y^2} \leq \frac{c\epsilon}{(\log(1/\epsilon))^2}.
\]
\end{lemma}

\begin{proof}[{\bf Proof}]
By Jensen's inequality we get
\begin{align}\label{eq:jensen}
\E{Y^2} \leq \int_{b}^{d}\E{(\vol (V_a(I,2\epsilon)^2} \,da.
\end{align}
We will first show that for all $\delta>0$ and all $a\in \R$
\begin{align}\label{eq:volumesquared2}
\E{(\vol (V_{a}(I,\delta))  )^2} \leq 2\E{\vol (V_{a}(I,\delta))}^2 .
\end{align}

For all $x$ define $\tau_x = \inf\{t\in I: X_t+at \in \B(x,\delta)\}$. We then have
\begin{align}\label{eq:volumesquared}
\E{(\vol (V_{a}(I,\delta))  )^2}  & = \int_{\R} \int_{\R} \pr{\tau_x <\infty, \tau_y <\infty}\, dx \,dy = 
2 \int_{\R}\int_{\R} \pr{\tau_x \leq \tau_y<\infty} \,dx \,dy \nonumber
\\
\nonumber
&= 2\int_{\R} \pr{\tau_x<\infty} \int_{\R}\prcond{\tau_x\leq \tau_y<\infty}{\tau_x <\infty}\, dy\,dx \\
&= 2\int_{\R} \pr{\tau_x<\infty} \econd{\vol (V_{a}(I\cap[\tau_x,1],\delta)) }{\tau_x<\infty} \,dx.
\end{align}
Since $(X_u+au)$ is c{\`a}dl{\`a}g and the filtration~$\til{\F}$ is right continuous, it follows from Claim~\ref{cl:stoppingtime} that $\tau_x$ is a stopping time. 
By the stationarity, the independence of increments and the c{\`a}dl{\`a}g property of $X$, we get that $X$ satisfies the strong Markov property (see~\cite[Proposition~I.6]{Bertoin}).
In other words, on the event $\{\tau_x<\infty\}$, the process $(X(\tau_x+t))_{t\geq 0}$ is c{\`a}dl{\`a}g and has independent and stationary increments.
Thus we deduce
\[
\econd{\vol(V_{a}(I\cap [\tau_x,1],\delta))}{\tau_x<\infty} \leq \E{\vol(V_{a}(I,\delta))},
\]
and this finishes the proof of~\eqref{eq:volumesquared2}

Applying Lemma~\ref{cl:voltr} with $t=\sqrt{\epsilon}$ and $r=\epsilon$ gives that there exists a constant $c$ so that for all $\epsilon$ sufficiently small and for all $a\in [-2,2]$
\[
\E{\vol(V_{a}(I,2\epsilon)} \leq \frac{c\sqrt{\epsilon}}{\log(1/\epsilon)}.
\]
Therefore from~\eqref{eq:jensen} and~\eqref{eq:volumesquared2}, since the above bound is uniform over all $a\in [-2,2]$, we deduce that for all $\epsilon$ sufficiently small
\[
\E{Y^2} \leq \frac{c'\epsilon}{(\log(1/\epsilon))^2}
\]
and this completes the proof.
\end{proof}

\begin{proof}[{\bf Proof of Theorem~\ref{thm:Cauchyproc}}]

It is easy to see that $\cup_{k=0}^{3}e^{ik\pi/4}\Lambda$ is a Kakeya set. 
Indeed, if we fix $t$ and we vary $a$, then we see that $\Lambda$ contains all directions from $0$ up to $45^\circ$ degrees. It then follows that the set $\cup_{k=0}^{3}e^{ik\pi/4} \Lambda$ contains a unit line segment in all directions.

It remains to show that there is a constant $c$ so that almost surely for all $\epsilon$ sufficiently small
\begin{align}\label{eq:goalvola}
\vol(\Lambda(\epsilon)) \leq \frac{c}{\log(1/\epsilon)}.
\end{align}
Note that it suffices to show the above inequality for $\epsilon$ which goes to $0$ along powers of $4$.
%
It is easy to see that for all $\epsilon>0$ we have
\[
\Lambda(\epsilon) \subseteq \bigcup_{-\epsilon\leq a\leq 1+\epsilon} \{a\} \times V_a([0,1],2\epsilon).
\]
Indeed, let $(x,y) \in \Lambda(\epsilon)$. Then $x\in [-\epsilon,1+\epsilon]$ and $(x-b)^2 + (y-(X_t + bt))^2 <\epsilon^2$ for some $b,t \in [0,1]$.
By the triangle inequality and since $t\in [0,1]$, we get
\[
|y-(X_t+xt) | \leq |y-(X_t+bt)| + |(b-x)t| \leq 2\epsilon.
\]
Take $\epsilon = 2^{-2n}$.
Thus in order to show~\eqref{eq:goalvola}, it suffices to prove that almost surely for all $n$ sufficiently large we have
\begin{align}\label{eq:newgoalvol}
\vol\left(\bigcup_{-2^{-2n}\leq a\leq 1+2^{-2n}} \{a\} \times V_a([0,1],2^{-2n+1})\right) \leq \frac{c}{\log(2^{2n})}.
\end{align}
For $j=1,\ldots,2^n$ define $I_j = [(j-1)2^{-n},j2^{-n}]$. Since $V_a([0,1],2^{-2n+1}) = \cup_{i\leq 2^n} V_a(I_i,2\epsilon)$ for all $a$, writing  
\[
Y_i = \vol\left(\bigcup_{-2^{-2n}\leq a\leq 1+2^{-2n}} \{a\} \times V_a(I_i,2^{-2n+1})\right) =  \int_{-2^{-2n}}^{1+2^{-2n}} \vol(V_a(I_i,2^{-2n+1})) \,da
\] 
we have by the subadditivity of the volume that
\[
\vol\left(\bigcup_{-2^{-2n}\leq a\leq 1+2^{-2n}} \{a\} \times V_a([0,1],2^{-2n+1})\right) \leq \sum_{i=1}^{2^n} Y_i.
\]
Hence it suffices to show that almost surely eventually in $n$
\begin{align}\label{eq:lastgoalvoly}
\sum_{i=1}^{2^n} Y_i \leq \frac{c}{\log(2^{2n})}.
\end{align}
Since $X$ has independent and stationary increments, it follows that the random variables $Y_i$ are independent and identically distributed. 
From Lemma~\ref{lem:yestimate} we obtain that $\var(Y_i) \leq c4^{-n}(\log(2^{2n}))^{-2}$ for all $n$ sufficiently large and thus we conclude by independence that eventually in $n$
\[
\var\left(\sum_{i=1}^{2^n} Y_i\right) \leq \frac{c 2^{-n}}{(\log(2^{2n}))^2}.
\]
From Chebyshev's inequality we now get
\[
\pr{\left|\sum_{i=1}^{2^n} Y_i - \E{\sum_{i=1}^{2^n}Y_i } \right| \geq \frac{1}{\log (2^{2n})} } \leq \frac{c}{2^{n}},
\]
which is summable and hence Borel Cantelli now gives that almost surely for all $n$ large enough
\[
\sum_{i=1}^{2^n} Y_ i \leq \E{\sum_{i=1}^{2^n} Y_i} + \frac{1}{\log(2^{2n})}.
\]
Using Lemma~\ref{cl:voltr} gives that $\E{Y_i} \leq c 2^{-n} (\log(2^{2n}))^{-1}$, and hence this  together with the inequality above finishes the proof.
\end{proof}

We now show that our construction is optimal in terms of boundary size up to a constant factor.

\begin{proposition}\label{pro:besicvol}
Let $K$ be a Kakeya set, i.e.\ a set that contains a unit line segment in all directions. Then for all $\epsilon$ sufficiently small
\[
\vol(K(\epsilon)) \geq \frac{1}{2\log(1/\epsilon)}.
\]
\end{proposition}

\begin{proof}[{\bf Proof}]

Let $\epsilon>0$ and $n=\lfloor \epsilon^{-1}\rfloor$. Suppose $x_i,v_i \in \R^2$ for all $i=1,\ldots,n$ are such that the unit line segments
$\ell_i = \{ x_i + tv_i : t\in [0,1]\}$ for $i=1,\ldots,n$ are contained in the set $K$ and satisfy 
\[
\sphericalangle (\ell_{i-1},\ell_i) = \frac{\pi}{n} \text{ for all } i=1,\ldots,n \ \text{ and }  \ \sphericalangle (\ell_{1},\ell_n) = \frac{\pi}{n}.
\]
For each $i$ take $w_i \perp v_i$ and define the set 
\[
\til{\ell}_i(\epsilon) = \{ x_i + tv_i +s w_i: t\in [0,1], s\in [-\epsilon,\epsilon]\}
\]
as in Figure~\ref{fig:line}. 
\begin{figure}
\begin{center}
\epsfig{file=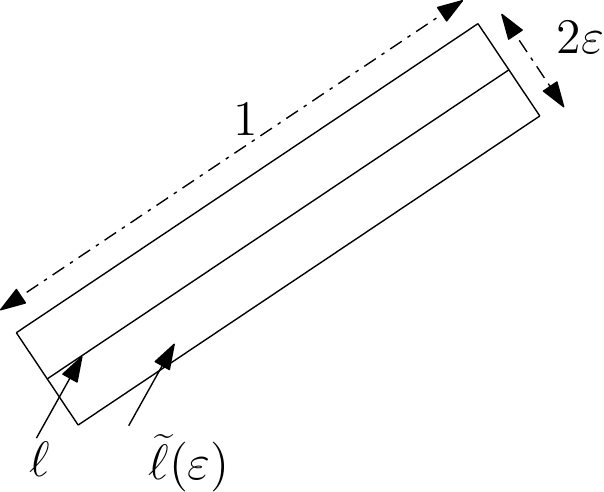, width=7 cm}
\caption{\label{fig:line} Line $\ell$ and rectangle $\til{\ell}(\epsilon)$}
\end{center}
\end{figure}
Then it is clear that
\begin{align}\label{eq:inclusion}
\til{\ell}_i(\epsilon) \subseteq K(\epsilon) \ \text{ for all } i.
\end{align}
For all $i=1,\ldots,n$ we define a function $\Psi_i:\R^2 \to \{0,1\}$ via $\Psi_i(x) = \1(x \in \til{\ell}_i(\epsilon))$ and let $\Psi = \sum_{i=1}^{n} \Psi_i$. Then from~\eqref{eq:inclusion} we obtain $\{x: \Psi(x)>0 \} \subseteq K(\epsilon)$,
and hence it suffices to show that for all $\epsilon$ sufficiently small
\begin{align}\label{eq:psigoal}
\vol(\{\Psi >0\}) \geq \frac{1}{2\log(1/\epsilon)}.
\end{align}
By Cauchy-Schwarz we get
\begin{align}\label{eq:cauchysch}
\vol(\{\Psi>0\}) \geq \frac{\left( \int_{\R^2} \Psi(x) \, dx    \right)^2}{\int_{\R^2}\Psi^2(x) \,dx }.
\end{align}
By the definition of the function $\Psi$ we immediately get that
\begin{align}\label{eq:firstmomentpsi}
\int_{\R^2} \Psi(x) \,dx = \sum_{i=1}^{n} \int_{\R^2} \1(x\in \til{\ell}_i(\epsilon))\,dx = 2\epsilon n.
\end{align}
Since $\Psi_i^2 = \Psi_i$ for all $i$, we have
\[
\int_{\R^2}\Psi^2(x) \,dx = \int_{\R^2} \Psi(x) \,dx + 2 \sum_{i=1}^{n}\sum_{k=1}^{n-i} \vol(\til{\ell}_i(\epsilon) \cap \til{\ell}_{i+k}(\epsilon)).
\]

\begin{figure}
\begin{center}
\epsfig{file=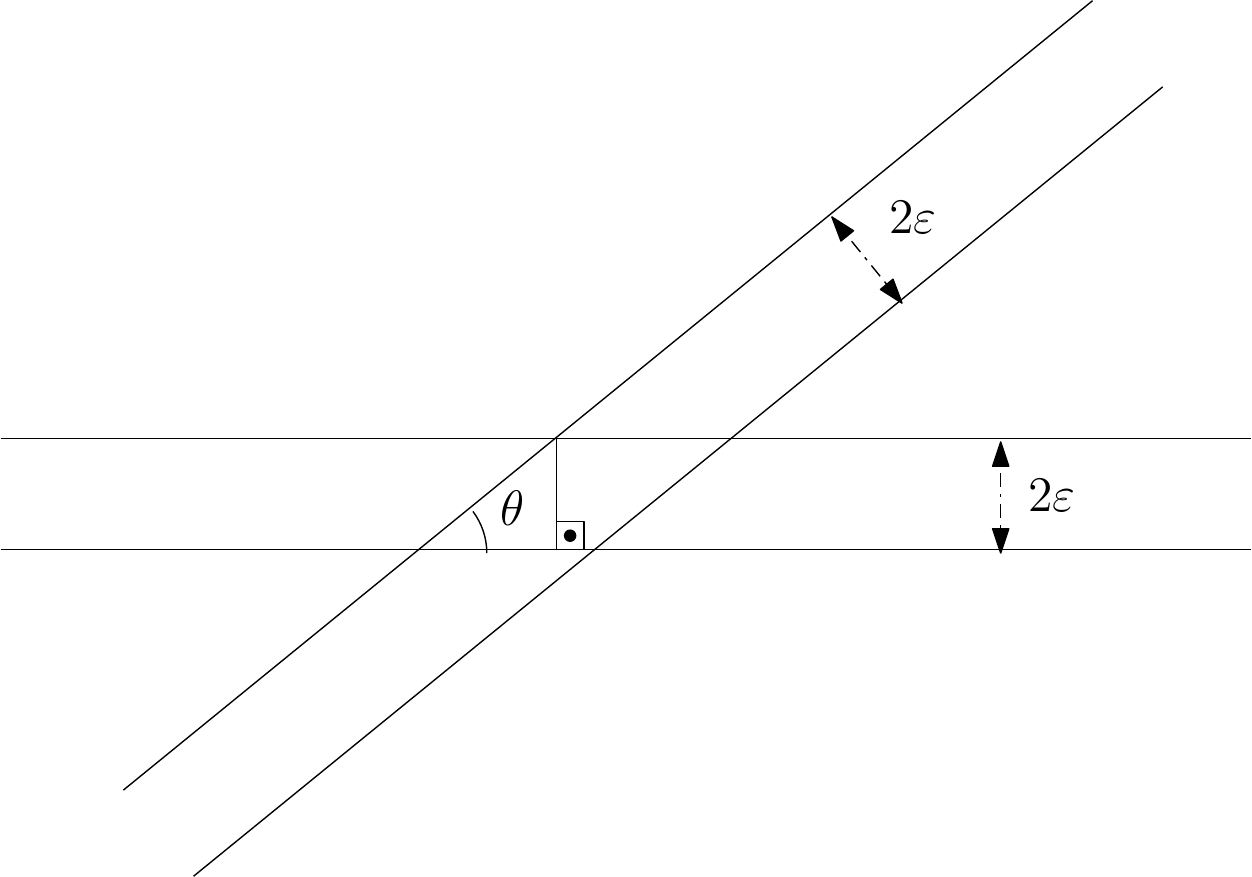, width=8cm}
\caption{\label{fig:intersection} Intersection of two infinite strips}
\end{center}
\end{figure}

The angle between the lines $\ell_i$ and $\ell_{i+k}$ is $k\pi/n$. From Figure~\ref{fig:intersection} we see that if $k\pi/n \leq \pi/2$, then 
\[
\vol(\til{\ell}_i(\epsilon) \cap \til{\ell}_{i+k}(\epsilon)) \leq \frac{4\epsilon^2}{\sin(k\pi/n)} \leq \frac{2\epsilon^2 n}{k},
\]
while if $k\pi/n >\pi/2$, then 
\[
\vol(\til{\ell}_i(\epsilon) \cap \til{\ell}_{i+k}(\epsilon)) \leq 
\frac{4\epsilon^2}{\sin(\pi-k\pi/n)} \leq \frac{2\epsilon^2 n}{n-k}.
\]
Hence using these two inequalities we deduce that
\begin{align*}
\sum_{i=1}^{n}\sum_{k=1}^{n-1} \vol(\til{\ell}_i(\epsilon) \cap \til{\ell}_{i+k}(\epsilon)) &\leq \sum_{i=1}^{\lfloor n/2 \rfloor} \left(\sum_{k=1}^{\lfloor n/2 \rfloor} \frac{2\epsilon^2 n}{k} + \sum_{k=\lfloor n/2\rfloor +1}^{n-i} \frac{2\epsilon^2n}{n-k}  \right)+\sum_{i=\lfloor n/2 \rfloor +1}^{n} \sum_{k=1}^{n-i} \frac{2\epsilon^2 n}{k} \\
& \leq 3\epsilon^2 n^2 \log n + 3\epsilon^2 n^2 .
\end{align*}
Thus putting all these bounds in~\eqref{eq:cauchysch} we obtain
\[
\vol(\{\Psi >0\}) \geq \left( \frac{3}{2}\log n + \frac{3}{2} + \frac{1}{2\epsilon n}    \right)^{-1}.
\]
Since $n= \lfloor \epsilon^{-1}\rfloor$, we conclude that for $\epsilon$ sufficiently small
\[
\vol(\{\Psi >0 \}) \geq \frac{1}{2\log(1/\epsilon)}
\]
and this finishes the proof.
\end{proof}

\section*{Acknowledgement}

We thank Abraham Neyman for useful discussions. 
\newline
Yakov Babichenko's work is supported in part by ERC grant 0307950, and by ISF grant 0397679.
Ron Peretz's work is supported in part by grant \#212/09 of the Israel Science Foundation and by the Google Inter-university center for Electronic Markets and Auctions.
Peter Winkler's work was supported by Microsoft Research, by a Simons Professorship at MSRI, and by NSF grant DMS-0901475.   We also thank MSRI, Berkeley, where part of this work was completed.

\bibliographystyle{plain}
\bibliography{biblio}

\end{document}